\documentclass[12pt]{article}
\usepackage{latexsym}
\usepackage{theorem}
\usepackage{graphicx}
\usepackage{amsmath}
\usepackage{color}
\usepackage{xcolor}
\usepackage{amsfonts}
\usepackage{natbib}
\usepackage{soul}
\usepackage{hyperref}
\usepackage{enumerate}

\theorembodyfont{\rmfamily}
\newtheorem{theorem}{Theorem}

\newtheorem{lemma}[theorem]{Lemma}

\newtheorem{definition}[theorem]{Definition}

\theoremstyle{break}

\newenvironment{proof}{\paragraph{Proof:}}{\hfill$\square$}

\title{A Review of Minimum Cost \\ Box Searching Games}

\author{Thomas Lidbetter\thanks{Department of Management Science and Information Systems, Rutgers Business School, Newark, NJ 07102, tlidbetter@business.rutgers.edu} }

\date{}

\begin{document}

\maketitle

\begin{abstract}
\noindent We consider a class of zero-sum search games in which a Hider hides one or more targets among a set of $n$ boxes. The boxes may require differing amounts of time to search, and detection may be imperfect, so that there is a certain probability that a target may not be found when a box is searched, even when it is there. A Searcher must choose how to search the boxes sequentially, and wishes to minimize the expected time to find the target(s), whereas the Hider wishes to maximize this payoff. We review some known solutions to different cases of this game.
\end{abstract}



\section{Introduction}

The origin of the field of {\em search games} is often traced back to \cite{isaacs}, where the author introduced a so-called {\em simple search game}, played in an arbitrary subset $\mathcal{R}$ of Euclidean space.  The formulation is very general, and involves a Hider choosing a hiding point in $\mathcal{R}$, and the Searcher choosing some kind of trajectory, which explores the whole of $\mathcal{R}$. The payoff of the game is the time taken for the Hider to be discovered, and the game is zero-sum, with the Hider being the maximizer and the Searcher being the minimizer. 

Isaac's game later inspired \cite{gal1979search} to formally introduce search games on networks, played either with an immobile Hider, as in \cite{isaacs}, or with a mobile Hider. This led to a rich theory of search games on networks, which can be found in \cite{alpern2003theory} and in the more recent reworking of \cite{alpern2011new}.

But search games can be traced back a little further. Indeed, prior to the work of \cite{isaacs}, a discrete search game was introduced by \cite{bram19632}. In this game, a Hider chooses one of $n$ boxes in which to hide a target, and a Searcher opens the boxes one-by-one until finding the target, with the aim of minimizing the total expected number of boxes to be opened. In this article, we will present this game formally, and review what is known about it, including special cases and generalizations, including a variation where each box takes a given amount of time to search.

\section{Game Definition and Preliminaries}

We consider a zero-sum game played between a Hider (the maximizer) and a Searcher (the minimizer). There is a set of $n$ boxes, which we denote $\mathcal{B} \equiv  \{1,2,\ldots,n\}$. Each box $j$ has a {\em detection probability} $q_j \in (0,1]$, which is the probability, conditional on the target lying in box $j$, that it is detected whenever that box is searched. The outcome of each search is independent of all previous search outcomes. The Hider chooses a box to hide the target in, so the Hider's pure strategy set is $\mathcal{B}$. Thus, his mixed strategy set is
\[
\Delta(\mathcal B) \equiv \Bigl\{ (p_1,\ldots,p_n): p_j \ge 0 \text{ for } j \in \mathcal{B} \text{ and } \sum_{j=1}^n p_j =1 \Bigl\}.
\] 
A pure strategy for the Searcher is an infinite sequence $\xi$ of boxes, so her pure strategy set may be denoted $\mathcal{B}^\infty$. For mixed Searcher strategies, we specify that the support is countable, in line with \cite{clarkson2023classical}, so a mixed strategy for the Searcher is specified by a function $\theta: \mathcal{B}^\infty \rightarrow [0,1]$, where $\{\xi \in \mathcal{B}^\infty: \theta(\xi) >0\}$ is countable, and $\sum_{\xi} \theta(\xi) =1$. The countability of $\{\xi \in \mathcal{B}^\infty: \theta(\xi) >0\}$ ensures that this sum is well-defined. 

We add one more set of parameters to the game, introduced later by \cite{clarkson2023classical}: each box $j \in \mathcal{B}$ has a {\em search time} $t_j >0$, which is the time it takes to search box $j$.

The payoff of the game is the total expected time to find the target. For a Hider strategy~$j$ and a Searcher strategy $\xi$, the payoff, which we refer to as the {\em search time}, is denoted $u(j, \xi)$. We extend the definition of $u$ to mixed strategies, so that for a Hider mixed strategy $\mathbf{p}$ and a Searcher mixed strategy $\theta$, the expected payoff (or {\em expected search time}) is given by
\[
u(\mathbf{p}, \theta) \equiv \sum_{j=1}^n \sum_{\theta \in \mathcal{B}^\infty} p_j \theta(\xi) u(j, \xi).
\]
Since the Searcher's pure strategy set has uncountably infinite cardinality, it is not immediately obvious that a minimax theorem holds and that the game has a value or optimal strategies. However, since the Hider's pure strategy set is finite and the payoff function is bounded below by $0$, it follows from standard results on semifinite games -- see, for example, \cite[Chapter 13]{ferguson2020course} -- that the game has a value $V$ given by
\[
V = \max_{\mathbf p \in \Delta^n} \inf_{\xi \in \mathcal{B}^\infty} u(\mathbf p, \xi) = \inf_{\theta} \max_{j \in \mathcal{B}} u(j, \theta).
\]
It also holds that the Hider has an optimal (max-min) strategy and the Searcher has an $\varepsilon$-optimal strategy: that is, for any $\varepsilon >0$, she has a strategy $\theta$ such that $u(j,\theta) \le V + \varepsilon$ for all $j \in \mathcal{B}$. In fact, it is possible to show that the Searcher has an optimal strategy that mixes between at most $n$ pure strategies -- see \cite{clarkson2023classical} -- but the proof of this is beyond the scope of this article.

It is important to understand how to calculate a best response for the Searcher against a fixed Hider mixed strategy $\mathbf{p}$: that is, a strategy $\xi$ that minimizes $u(\mathbf{p}, \xi)$. This is an interesting problem in itself, and was first solved by Blackwell \cite[reported in][]{matula1964periodic}, and also independently by \cite{black1965discrete}. We present the latter solution, which is derived using a particularly elegant graphical method. 

We fix the Hider strategy $\mathbf{p}$, and for a Searcher strategy $\xi$, we write $T_k(\xi)$ for the total time taken for the first $k$ looks according to $\xi$, and we write $P_k(\xi)$ for the probability that the target is found on or before the $k$th look. Then the expected search time can be written
\[
u(\mathbf{p}, \xi) = \sum_{k=1}^\infty (T_k(\xi) - T_{k-1}(\xi))(1 - P_{k-1}(\xi)),
\]
as depicted in Figure~1. The shaded area is the expected search time.

\begin{figure}[h]
\center
\includegraphics[width=10cm]{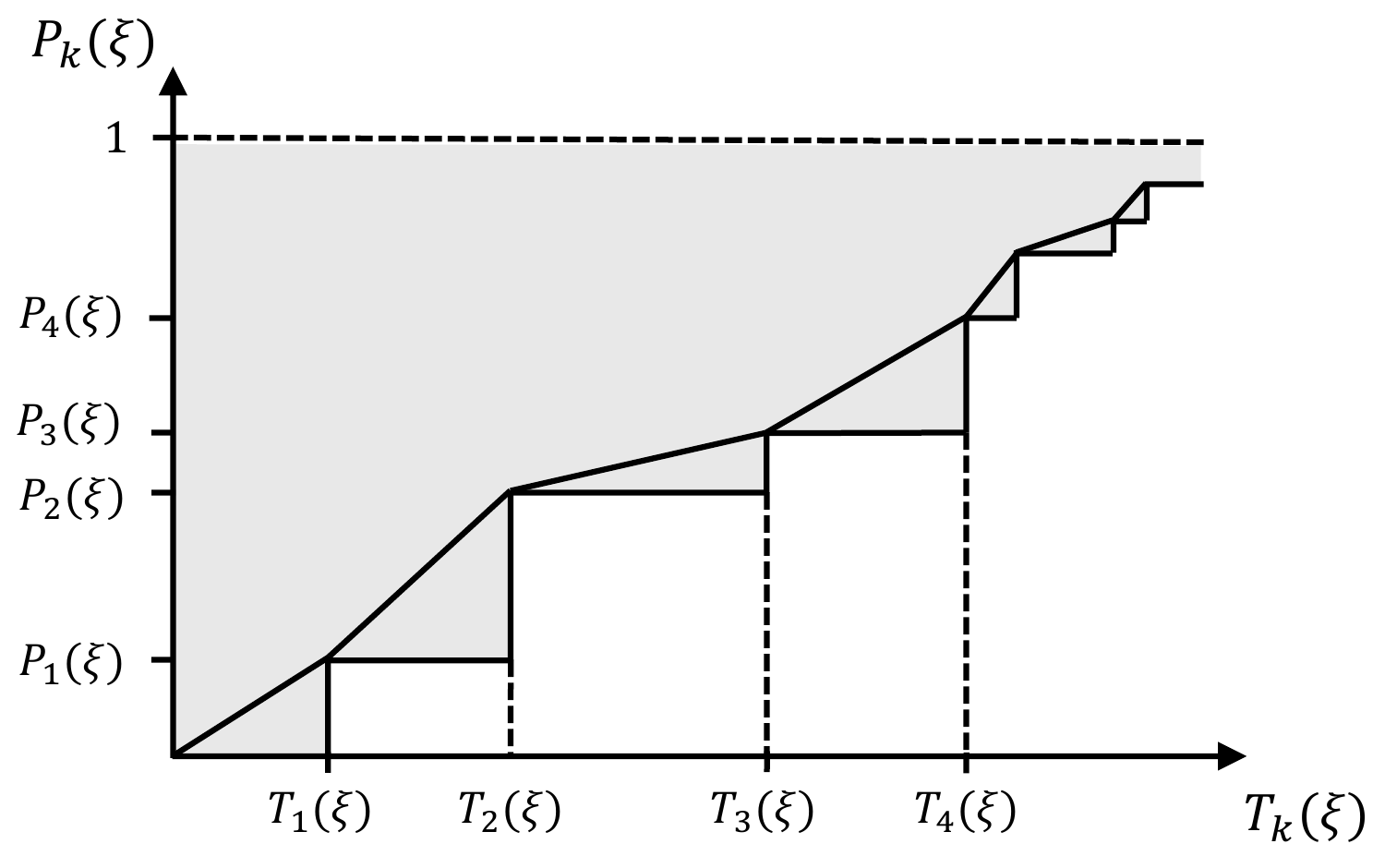} 
\caption{Graphical representation of the expected search time of a search $\xi$ against a fixed Hider strategy $\mathbf p$.}
\label{fig:plot}
\end{figure}

We can think of the shaded area in Figure~\ref{fig:plot} as the union of rectangles placed side by side, each corresponding to a look in a particular box. Each rectangle is composed of a triangle and a trapezium, as depicted in the figure. Consider the rectangle corresponding to the $m_i$th look in box $i$. The base of the corresponding triangle has length $t_i$, and the height of the triangle is equal to the probability the target is found on that look, which is $p_i (1-q_i)^{m_i-1} q_i$. Since all searches $\xi$ with finite expected cost must search each box an infinite number of times, all such searches must contain the same set of triangles in the plot of Figure~1, so the expected search time is only determined by the area of the trapezia. 

Clearly, this area is minimized if the triangles are ordered in non-increasing order of their slopes 
\[
\psi_{j,m_j} \equiv \frac{p_j (1-q_j)^{m_j-1} q_j }{ t_j}.
\] 
It is important to check that the search that arranges the triangles in this way is feasible: this follows from the fact that $\psi_{j,m_j}$ is decreasing in $m_j$. Note that if all the detection probabilities $q_j$ are equal to $1$ (so that each box only needs to be checked once), then this optimal policy boils down to ordering the boxes in non-increasing order of $p_j/t_j$. This optimal search policy for the finite problem had previously been found by \cite{bellman} (Chapter III, Exercise 3, p.90), and is equivalent to {\em Smith's Rule} for minimizing the sum of weighted completion times on one machine \citep{smith1956various}.

This policy for an optimal search has a nice  interpretation. Indeed, suppose the next search of box $j$ will be the $m_j$th search of that box, for each $j \in \mathcal{B}$. Then, the conditional probability $p_j'$ that the target is in box $j$, given that it has not yet been found, is given by
\begin{align}
p_j' = \frac{p_j(1-q_j)^{m_j-1}}{ \sum_{i=1}^n p_i (1-q_i)^{m_i-1}}, \label{eq:p}
\end{align}
for each $j \in \mathcal{B}$. Hence, the index $\psi_{j, m_j}$ can be written $\psi_{j,m_j} = \lambda p_j' q_j/t_j$, where $\lambda$ is a constant that depends on $m_1,\ldots,m_n$, and is independent of $j$. The index $p_j' q_j/t_j$ is the ratio of  the conditional probability of finding the target in box $j$ on the next look, given that it has not yet been found, to the time taken to search that box. So an optimal search has the property that,  any point in the search, the next box to be searched should be one that maximizes this ratio.

Observe that there may be many optimal strategies if at some point in the search, the index $\psi_{j,m_j}$ is not maximized by a unique $j$. For example, consider the case $n=2$, $t_1=t_2=1$, $p_1=p_2=1/2$ and $q_1=q_2=1/2$. Clearly, a search strategy $\xi=(a_1,a_2,\ldots)$ is optimal if and only if it has the property that $a_{2j-1} \neq a_{2j}$ for each $j=1,2,\ldots$. Indeed, the Searcher is indifferent between opening boxes $1$ and $2$ at the beginning of the search, then after Bayesian updating the new hiding probabilities are $1/3$ and $2/3$ in some order, so the Searcher opens a different box next. Updating the probabilities again, the distribution returns to $(1/2,1/2)$.

A natural Hider strategy to consider in the game is the one that makes all the indices $\psi_{j,1}$ equal, so that at the beginning of the search, the Searcher is indifferent between looking in all the boxes. We call this strategy Hider's {\em equalizing strategy}, defined formally below.
\begin{definition}
The Hider's equalizing strategy $\mathbf{p}^*$ is given by
\begin{align}
p^*_j = \frac{t_j/q_j}{\sum_{i=1}^n t_i/q_i}, 
\label{eq:equalizing}
\end{align}
for each $j \in \mathcal{B}$.
\end{definition}
One might conjecture that the Hider's equalizing strategy is always optimal. It turns out that this is true in some special cases.

\section{Equal Search Times and Detection Probabilities}

As a warm-up, we consider our game in the case that all the search times, $t_j$ are equal to $1$ and all the detection probabilities are equal to some $q \in (0,1]$. We will denote the value of the game by $V_q$ in this section.

If $q$ is equal to $1$, then the game is finite, and by symmetry it is clear that it is optimal for the Hider to choose each of his $n$ strategies with probability $1/n$ (which is the equalizing strategy, $\mathbf p^*$), and for the Searcher to choose each of her $n!$ strategies with probability $1/n!$. We can compute the value $V_1$ of the game explicitly by calculating the expected search time of an arbitrary ordering of the boxes against $\mathbf p^*$:
\[
V_1= \sum_{j=1}^n \frac{1}{n} \cdot j = \frac{n+1}{2}.
\]
In fact, there is a neater optimal strategy for the Searcher, which may be thought of as follows. The boxes are arranged on a circle, clockwise from $1$ to $n$. For each $i \in \mathcal B$, Let $s_i$ be the ordering of the $n$ boxes that starts with box $i$ and proceeds clockwise through the remaining boxes, so that $s_i=(i,i+1,\ldots,n,1,2,\ldots,i-1)$. Now consider the Searcher strategy $\xi^*$ that chooses each of the $n$ strategies $s_i$ with equal probability. The expected search time of this strategy against an arbitrary Hider strategy $j$ is
\[
u(j, \xi^*) =  \sum_{i=1}^j \frac{1}{n} \cdot (j-i+1) + \sum_{i=j+1}^n \frac{1}{n} \cdot (n-i+j+1) = \frac{n+1}{2}.
\]
Now suppose the detection probabilities are all equal to some $q \in (0,1)$. In this case, it is again clear that, by symmetry, $\mathbf p^*$ is optimal. The value can be calculated by considering an optimal response to $\mathbf p^*$. Such a best response must start by searching all the boxes in an arbitrary order. Subsequently, the conditional hiding distribution reverts back to the initial distribution $\mathbf p^*$, so a best response continues repeatedly searching all $n$ boxes in any order.  Let us call each of these $n$ searches of the boxes ``rounds'', and note that the probability that the target is found in round $k$ is $q(1-q)^{k-1}$ for $k=1,2,\ldots$. Given it is found in round $k$, the expected search time is the sum of  $(k-1)n$ (the time spent searching in the first $k-1$ rounds) and $V_1=(n+1)/2$ (the expected time spent searching in round $k$). Hence, the value $V_q$ of the game is given by
\begin{align*}
V_q &= \sum_{k=1}^\infty q(1-q)^{k-1} \left((k-1)n + \frac{n+1}{2} \right) \\
&= \frac{n}{q} - \frac{n-1}{2}. 
\end{align*}
An optimal Searcher strategy can be obtained by expanding the definition of the optimal Searcher strategy $\xi^*$ for the finite case. Whereas in the case of perfect detection $\xi^*$ chooses some $s_i$ at random, for imperfect detection $s_i$ is chosen at random and repeated indefinitely. Then, the probability the target is found given the target is found in round $k$ is again $q(1-q)^{k-1}$, and the conditional expected search time is $(k-1)n+V_1$, given that the target is found in round $k$. Hence, the same calculation as above shows that this Searcher strategy ensures the expected search time is precisely $V_q$ against any Hider strategy.

The solution of this special case can be attributed to \cite{ruckle1991discrete}.

An alternative optimal Searcher strategy makes the random choice of $s_i$ at the beginning of each round of search.

\section{Equal Detection Probabilities}

We now relax the assumption of equal search times, but still assume that the detection probabilities are all equal to some $q \in (0,1]$. In this section, we will denote the value of the game by $v_q$. 

As in the previous section, we will start with the finite case $q=1$. The solution of this case was first discovered by \cite{condon2009algorithms}. Alternative solutions can be found as a special cases of solutions of more general games in \cite{lidbetter2013search} and \cite{alpern2013mining}. 

We first calculate the lower bound on the value given by the Hider's equalizing strategy $\mathbf p^*$, given by $p^*_j = t_j/t(\mathcal{B})$, where $ t(\mathcal{B}) = \sum_{j \in \mathcal{B}} t_j$. Suppose the Searcher looks in the boxes in the order $1,2,\ldots,n$. Then the expected search time is
\[
\sum_{j=1}^n p^*_j \sum_{i=1}^j t_i = \frac{\sum_{j=1}^n  \sum_{i=1}^j t_i t_j }{t(\mathcal{B})} = \frac{t^2(\mathcal{B}) + t(\mathcal{B})^2}{2 t(\mathcal{B})},
\]
where $t^2(\mathcal{B}) = \sum_{j \in \mathcal{B}} t_j^2$.

We now show that the Searcher can match this bound on the expected search time by using a slightly more sophisticated version of the Searcher strategy of the previous section in the case of equal search times. This time, the strategy $\xi^*$ chooses $s_j$ with probability  $t_j/t(\mathcal{B})$. In this case, the expected search time to find a target located in  box $n$ is
\[
u(\xi^*, n) = \sum_{j=1}^n \frac{t_j}{t(\mathcal{B})} \sum_{i=j}^n t_i =  \frac{t^2(\mathcal{B}) + t(\mathcal{B})^2}{2 t(\mathcal{B})}.
\]
By a relabeling argument, the expected search time is the same whichever box the target lies in. Therefore, 
\[
v_1 =  \frac{t^2(\mathcal{B}) + t(\mathcal{B})^2}{2 t(\mathcal{B})}.
\]
Now suppose $q<1$. We will show that the Hider's equalizing strategy is still optimal.

At the start of the search, all the indices $\psi_{j,1}$ are equal, so a best response to $\mathbf p^*$ starts by opening all the boxes in some order. Subsequently, for each $j \in \mathcal{B}$, the conditional probability $p_j'$ of the target lying in box $j$ is, by~(\ref{eq:p}),
\[
p_j' = \frac{p^*_j(1-q)}{ \sum_{i=1}^n p^*_i (1-q)}= p^*_j.
\]
In other words, the hiding distribution reverts back to the Hider's equalizing strategy. Therefore, an optimal response to $\mathbf p^*$ repeatedly opens all the boxes in some order. By a similar calculation as for the case of equal costs, the expected search time of such a search is
\[
\sum_{k=1}^\infty q(1-q)^{k-1} \left((k-1)t(\mathcal{B}) + v_1 \right) = \frac{(1-q)t(\mathcal{B})}{q} +  \frac{t^2(\mathcal{B}) + t(\mathcal{B})^2}{2 t(\mathcal{B})}.
\]
An optimal Searcher strategy may also be obtained by choosing with probability $t_j/t(\mathcal{B})$ to repeatedly cycle through $s_j$. By an identical calculation to the one above we obtain an upper bound on the value that matches the lower bound. Hence,
\[
v_q = \frac{(1-q)t(\mathcal{B})}{q} +  \frac{t^2(\mathcal{B}) + t(\mathcal{B})^2}{2 t(\mathcal{B})}.
\]

\section{Non-Equal Detection Probabilities} \label{sec:non-equal}
Once we relax the assumption of non-equal detection probabilities, the game becomes a lot harder to solve. \cite{clarkson2024computing} recently developed an algorithm that converges to optimal strategies for both players, but finding some general expression for optimal strategies in closed form is an unrealistic goal. A more realistic goal is to find an algorithm whose output is optimal strategies in closed form, and in this section we give an idea of how such an algorithm may be constructed for the case $n=2$, by considering a concrete example.

One important observation is that if there is any hope of obtaining a closed form solution to an instance of the game, we must assume that there are some positive coprime integers $s_1,\ldots, s_n$ and some $q>0$ such that $(1-q_j)^{s_j} = q$ for all $j \in \mathcal{B}$. This assumption implies that for any fixed hiding distribution $\mathbf p$, after each location $j$ has been searched $s_j$ times, the hiding distribution returns to its initial state after Bayesian updating. In this case, there exist best responses to $\mathbf p$ that are cyclical, in the sense that after some finite time, they repeat. Without this assumption on the $q_j$, best responses cannot be cyclical, and therefore cannot be written in closed form.

We note that this assumption is not very restrictive, because for any $q_1,\ldots,q_n$, we can find some $s_1,\ldots,s_n$ such that $(1-q_j)^{s_j}, j=1,\ldots n$, are arbitrarily close to each other. However, depending on the desired closeness of approximation, $s_1,\ldots,s_n$ may have to be very large.

We note the following useful facts about optimal strategies, which were proved in \cite{clarkson2023classical}. For ease of exposition, we assume that the Searcher always has optimal strategies (which was also proved in \cite{clarkson2023classical}).

\begin{lemma} \label{lemma} Let $\mathbf p$ and $\theta$ be optimal strategies for the Hider and Searcher, respectively. 
\begin{enumerate}[(i)]
\item Every pure strategy chosen by $\theta$ with positive probability must be a best response to~$\mathbf p$.
\item The strategy $\mathbf p$ has full support: that is, $p_j >0$ for all $j \in \mathcal{B}$.
\item We have $u(j, \theta)=V$  for every $j \in \mathcal{B}$, where $V$ is the value of the game.
\end{enumerate}
\end{lemma}
\begin{proof} For part (i), first note that by the optimality of $\theta$, we have $u(\mathbf p, \theta) \le V$. Also, for any pure strategy $\xi \in Y$, we have $u(\mathbf p, \xi) \ge V$, by the optimality of $\mathbf p$. It follows that
\[
V \ge u(\mathbf p, \theta)= \sum_{\xi} \theta(\xi) u(\mathbf p, \xi) \ge V.
\]
Hence, both inequalities above hold with equality, and $u(\mathbf p, \xi) = V$ for each $\xi$ chosen by $\theta$ with positive probability. In other words, $\xi$ is a best response to $\mathbf p$. 

Turning to part (ii), suppose that $\mathbf p$ is an optimal Hider strategy that chooses some $j \in \mathcal B$ with probability $0$. Then any pure best response to $\mathbf p$ never visits location $j$. This is because if some $\xi$ did ever visit $j$, then by omitting all searches of $j$, we could obtain a sequence with a strictly smaller expected search time against $\mathbf p$. In this case, by part (i), any optimal mixed Searcher strategy $\theta$  must only choose with positive probability pure strategies that never visit $j$. But in that case, a best response to $\theta$ is to hide at location $j$, since that has infinite expected search time against $\theta$. This contradicts the optimality of $\theta$.

Part (iii) follows from the fact that $\mathbf p$ has full support, so each $j \in \mathcal{B}$ must be a best response to $\theta$.
\end{proof}

It is worth remarking that Lemma~\ref{lemma} actually holds for an arbitrary zero-sum semi-infinite game that has a value $V$ and optimal strategies. 

The only general class of instances of the game with non-equal detection probabilities for which the solution can be found is due to \cite{roberts1978search}. They showed that for $n=2$, if $(1-q_1)^s = (1-q_2)^{s+1}$ for some non-negative integer $s$, then the Hider's equalizing strategy is optimal if and only if $s \le 12$. The proof of this is not terribly instructive, so we omit it here.

In the remainder of this section we focus on the special case of $n=2$, and give an example to indicate how an optimal solution may be found in this case. 

\subsection{Example} \label{sec:ex}

In our example we take $q_1=1/2$ and $q_2=15/16$, so that 
\[
(1-q_1)^4 = (1-q_2) = \frac{1}{16}.
\]
We take the search times $t_1$ and $t_2$ to be $1$, although it turns out that lifting this restriction does not make matters any more complicated.

We first check whether the Hider's equalizing strategy $\mathbf p^*$ is optimal. Note that without explicitly writing down $\mathbf p^*$, we can write down a best response to the equalizing strategy. Either box can be chosen to be opened first. If box $2$ is opened first, then after Bayesian updating, the new hiding probabilities of box $1$ and $2$ are proportional to $p_1^*$ and $(1-p_2)p_2^*$, respectively, so that box $1$ becomes more attractive to the Searcher. Box $1$ must then be opened four times before the updated hiding probabilities are in proportion to $(1-q_1)^4p_1^*$ and $(1-q_2) p_2^*$, respectively. Since $(1-q_1)^4 = (1-q_2)$, this means the hiding distribution has returned  to $\mathbf p^*$. If, instead, box $1$ is opened first, then box $2$ must be opened next, followed by box $1$ three times, then, again, the distribution returns to $\mathbf p^*$. 

Write $S_1$ for the finite sequence $1,2,1,1,1$ and $S_2$ for the sequence $2,1,1,1,1$. Then the general form of a best response to $\mathbf p^*$ proceeds in blocks of $5$ searches, where each block is equal to either $S_1$ or $S_2$. Clearly, the best response that minimizes the expected search time given that the target is in box $1$ is $\xi^1 \equiv S_1, S_1, S_1,\ldots$, and the one that minimizes the expected search time given the target is in box $2$ is $\xi^2 \equiv S_2, S_2, S_2,\ldots$. 

Let us write $T_i(\xi)$ for the expected search time $u(i, \xi)$ of box $i$ under an arbitrary search strategy $\xi$, and we refer to $\mathbf{T}(\xi) \equiv (T_1(\xi),T_2(\xi))$ as the {\em performance vector} of $\xi$, borrowing terminology from \cite{queyranne1994polyhedral}. We can calculate the performance vectors of $\xi^1$ and $\xi^2$ as follows.
\[
T_1(\xi^2) = 1 + 1 + \frac{1}{2} \cdot 1 + \frac{1}{4} \cdot 1 + \frac{1}{8}\cdot 1 + \frac{1}{16} \cdot T_1(\xi^2) \quad \Rightarrow \quad T_1(\xi^2) = \frac{46}{15}.
\]
Similarly, 
\[
T_2(\xi^2)= \frac{4}{3}, \quad T_1(\xi^1) = \frac{38}{15}, \text{ and } T_2(\xi^1) = \frac{7}{3}.
\]
Note that any best response $\xi$ to $\mathbf p^*$ has the same expected search time against $\mathbf p^*$, and this expected search time is given by
\[
u(\mathbf p^*, \xi) = p_1^* T_1(\xi) + p_2^* T_2(\xi).
\]
Hence, the performance vector of any best response $\xi$ must lie on the straight line segment between the vectors $\mathbf{T}(\xi^1)$ and $\mathbf{T}(\xi^2)$ in the $xy$-plane. Conversely, any point on this line segment is the performance vector of some best response $\xi$ to $\mathbf p^*$, where $\xi$ may be obtained by an appropriate mixing of $\xi^1$ and $\xi^2$.

In Figure~\ref{fig:BR1} we display the performance vectors of all best responses to $\mathbf p^*$ in the $xy$-plane, along with the line $x=y$. By Lemma~\ref{lemma}, part (iii), any optimal mixed Searcher strategy must have a performance vector that lies on the line $x=y$. Since the line $x=y$ does not intersect the set of best responses to $\mathbf p^*$, we may conclude that $\mathbf p^*$ is not optimal for the Hider.

\begin{figure}[h]
\center
\includegraphics[width=8cm]{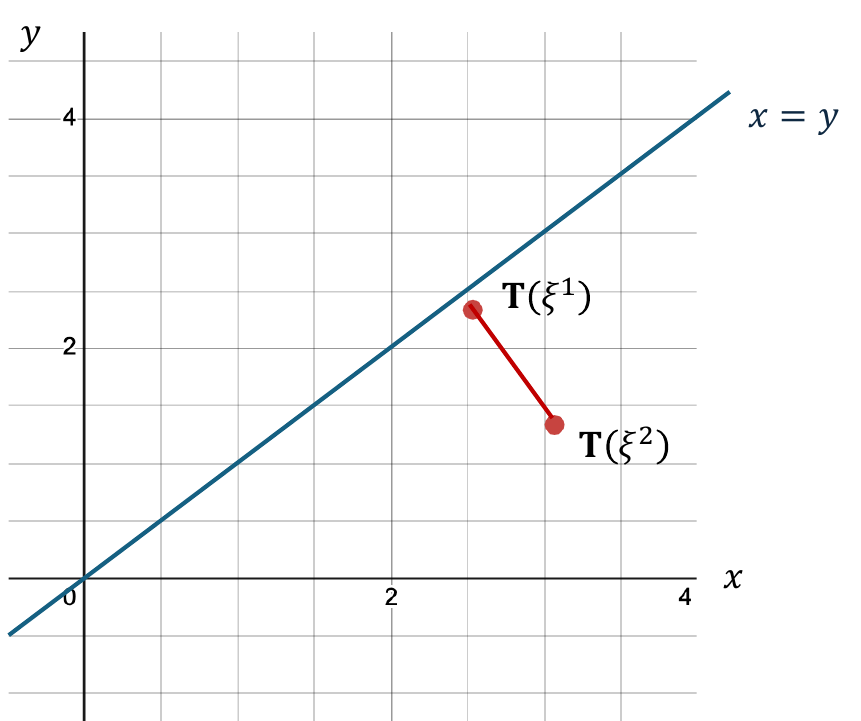} 
\caption{The set of performance vectors of strategies $\xi^1$ and $\xi^2$.}
\label{fig:BR1}
\end{figure}

To determine optimal strategies, it is helpful to understand the set of performance vectors of all Searcher strategies that are best responses to some Hider strategy, since any optimal Searcher strategy must necessarily be a best response to some Hider strategy. We refer to $\xi^1$ and $\xi^2$ as {\em extreme} search strategies. In general, an extreme search strategy is a best response to some Hider strategy that has the property that whenever there is an arbitrary choice between box 1 and box 2, it always chooses the same box. Clearly, for $j=1,2$, an extreme search strategy that always chooses box $j$ in the face of an arbitrary choice is the unique best response that minimizes the expected search time given that the target is in box $j$.

Consider the Hider strategy 
\[
\mathbf{p}' \equiv  \lambda \left( \frac{p_1^*}{1-q_1}, p_2^* \right),
\]
where $\lambda$ is a normalizing factor to ensure that $p'_1+p'_2=1$. The strategy $\mathbf p'$ is designed in such a way that a best response must first search box $1$, after which the distribution reverts to $\mathbf p^*$, after Bayesian updating. Therefore, the two extreme best responses to $\mathbf p'$ may be denoted $1,\xi^1$ and $1,\xi^2$: that is, $1$ followed by either $\xi^1$ or $\xi^2$. Note that $1,\xi^2$ is equal to $\xi^1$. Hence, the performance vectors of all best responses to $\mathbf p'$ form a line segment in the $xy$-plane from $\mathbf{T}(\xi^1)$ to $\mathbf{T}(1,\xi^1)$. 

It is easy to calculate the performance vector $\mathbf{T}(1,\xi^1)$ of $1,\xi^1$, using the following relations
\[
T_1(1,\xi^1) = p_1 + (1-p_1)(1+T_1(\xi^1)), \quad T_2(1, \xi^1) = 1+ T_2(\xi^1).
\]

Any Hider strategy $(p_1,p_2)$ with $p^*_1 < p_1 < p'_1$ has the unique best response $\xi^1$. In general, to obtain all possible best responses to some Hider strategy, it is easy to see that it is sufficient to consider only Hider strategies of the form
\[
\mathbf p = \lambda \left( p_1^*(1-q_1)^{r}, p_2^* \right),
\]
where $r$ is an integer and $\lambda$ is a normalizing factor. The performance vectors of the best responses to such strategies form a piecewise linear curve in the $xy$-plane, a portion of which is shown in Figure~\ref{fig:BR2}. We call this piecewise linear curve the {\em best response frontier}.

\begin{figure}[h]
\center
\includegraphics[width=9cm]{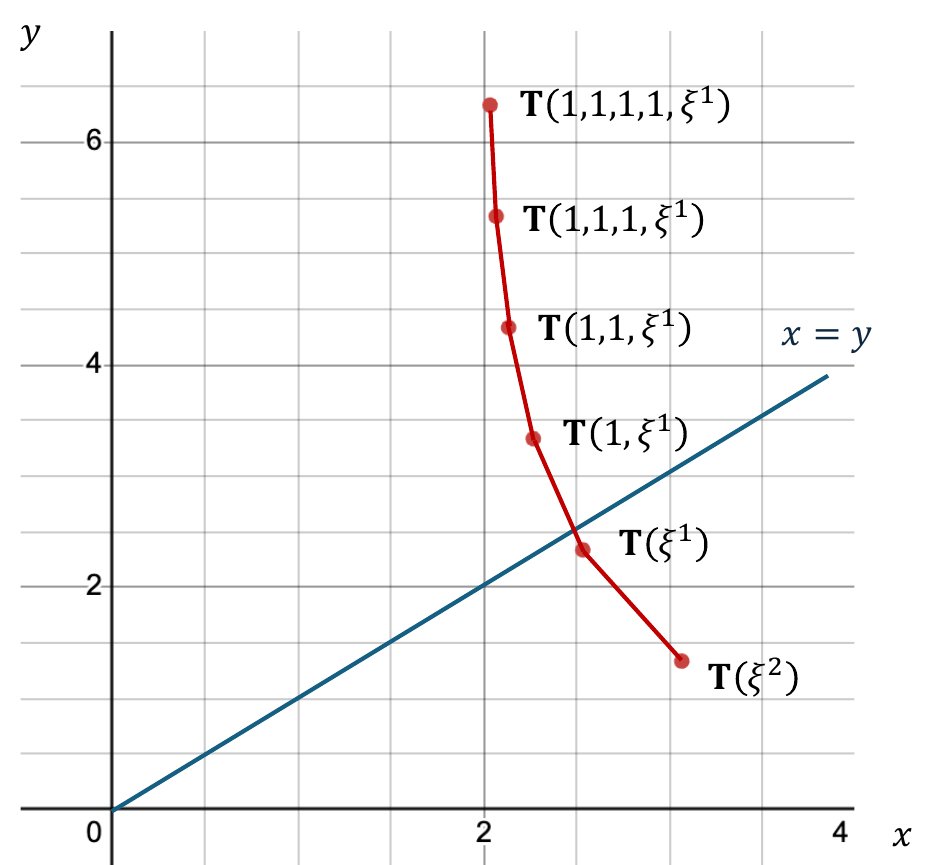} 
\caption{A subset of the best response frontier.}
\label{fig:BR2}
\end{figure}

It follows from Lemma~\ref{lemma}, part (iii), that we may determine an optimal strategy by finding the intersection of the $x=y$ with the best response frontier. In our example, it is clear that this intersection point is on the line segment between $\mathbf{T}(\xi^1)$ and $\mathbf{T}(1,\xi^1)$. Hence, the Hider strategy $\mathbf p'$ must be optimal, and it is optimal for the Searcher to mix between the strategies $\xi^1$ and $1,\xi^1$. We do not explicitly calculate the optimal mixture here, but it is easy to obtain by solving the following linear equation in $\beta$.
\[
\beta T_1(\xi^1) + (1-\beta) T_1(1,\xi^1) = \beta T_2(\xi^1) + (1-\beta) T_2(1,\xi^1).
\]
\subsection{Remarks}

In general, one approach to finding a solution of the game for $n=2$ would be to attempt to characterize the best response frontier, then find the point at which it intercepts the line $x=y$, as in the example of Subsection~\ref{sec:ex}. The example we chose is particularly simple, because we do not have to go too far on the best response frontier from the performance vectors of the best responses to $\mathbf p^*$ to find the point of intersection with the line $x=y$. It is not clear that finding the optimal strategies is so simple in general.

For $n \ge 3$, characterizing the best response frontier becomes highly non-trivial, and more work is needed if this approach is to be fruitful.

\section{Multiple Targets} \label{sec:multiple}

In this section we consider a generalization of the game, where we wish to find not one, but $k$ targets, for some integer $k=1,2,\ldots,n$. We consider only the case of perfect detection, where $q_1 = \cdots = q_n$, as solved in \cite{lidbetter2013search}.

A Hider's pure strategy set is now the set $\mathcal{B}^{(k)}$ of subsets of $\mathcal{B}$ of cardinality $k$, and a strategy for the Searcher is a {\em finite} sequence $\xi$, corresponding to a permutation of $\mathcal{B}$. The payoff function is the total time taken to find {\em all $k$} targets, and we continue to denote it by the function $u$.

Note that we could allow the Hider to hide multiple targets in the same box, but such a strategy would be weakly dominated by some strategy that hides the targets all in different boxes.

To define the Hider strategy $\nu$ that we will later see is optimal, we first introduce some notation. For a subset $\mathcal{A} \subseteq \mathcal{B}$ and $j=1,\ldots,n$, we define $S_j(\mathcal{A})$ as follows.
\[
S_j(\mathcal{A}) = \sum_{A \in \mathcal{A}^{(j)}} \pi(A),
\]
where $\pi(H) = \prod_{j \in H} t_j$ is the product of all the search times in $A$ and $\mathcal{A}^{(k)}$ is the collection of subsets of $\mathcal{A}$ of cardinality $k$. We write simply $S_j$ for $S_j(\mathcal{B})$. 

Now we define $\nu$ as follows. For a set $H \in \mathcal{B}^{(k)}$, let
\[
\nu(H) = \frac{\pi(H)}{S_k}.
\]
That is, $\nu(H)$ is chosen to be proportional to the product of all the search times of the boxes in $H$.

Let $\xi$ be an arbitrary Searcher strategy. By a relabeling argument, we may assume that $\xi=(1,2,\ldots,n)$. Clearly, at least $k$ boxes must be searched before all $k$ targets are found, and the probability that all $k$ targets have not been found before box $j \ge k+1$ is searched is
\[
1-\sum_{H \in [j-1]^{(k)}} \nu(H) = 1- \frac{S_k([j-1])}{S_k},
\]
where $[i]$ denotes $\{1,2,\ldots,i\}$.	Therefore, the expected search time is
\[
u(\nu,\xi) = t([k]) + \sum_{j=k+1}^n t_j \left( 1- \frac{S_k([j-1])}{S_k} \right) = t(\mathcal{B}) - \frac{ \sum_{j=k+1}^n t_j S_k([j-1])}{S_k}.
\]
It is easy to see that $\sum_{j=k+1}^n t_j S_k([j-1]) = S_{k+1}$, so 
\begin{align}
u(\nu,\xi) = t(\mathcal{B}) - \frac{ S_{k+1}}{S_k}. \label{eq:Vk}
\end{align}
Let $V_k = t(\mathcal{B}) - S_{k+1}/{S_k}$. Since $\xi$ was arbitrary, $V_k$ is a lower bound for the value of the game. We will show that $V_k$ is in fact the value of the game, so that the strategy $\nu$ is optimal for the Hider. 

For $H \in \mathcal{B}^{(k)}$, let $\xi_H$ be the Searcher strategy that starts by searching all the boxes in $H$, then searches the remaining boxes in a (uniformly) random order. (Note that the order of the first $k$ boxes does not matter, because at least $k$ boxes must be searched before all the targets are found). 

We consider the expected search time $u(H', \xi_H)$ of $\xi_H$ against an arbitrary Hider strategy $H' \in \mathcal{B}^{(k)}$. All the boxes in $H \cup H'$ must be searched, and each other box is searched with a fixed probability $\beta$, which only depends on the cardinality of $H \cup H'$. By a symmetrical argument, the same is true of the expected search time $u(H, \xi_{H'})$. Therefore, $u(H', \xi_H)$ and $u(H, \xi_{H'})$ must be equal, with
\begin{align}
u(H', \xi_H) = u(H, \xi_{H'}) = t(H \cup H') + \beta t(\mathcal{B} \setminus(H \cup H')). \label{eq:symm}
\end{align}
Since the precise value of $\beta$ is irrelevant for the argument, we do not calculate it explicitly. 

Now consider the Searcher mixed strategy $\theta$ that, for each $H \in \mathcal{B}^{(k)}$, chooses $\xi_H$   with probability $\nu(H)$. The expected search time of $\theta$ against an arbitrary $H' \in \mathcal{B}^{(k)}$ is
\[
u(H', \theta) = \sum_{H \in \mathcal{B}^{(k)}} \nu(H) u(H', \xi_H) =  \sum_{H \in \mathcal{B}^{(k)}} \nu(H) u(H, \xi_{H'}) =  u(\nu, \xi_{H'}) = V_k, 
\]
where the second inequality follows from (\ref{eq:symm}) and the final equality follows from (\ref{eq:Vk}).

We have shown that the strategies $\nu$ and $\theta$ are optimal and $V_k$ is the value of the game. Note that in this equilibrium, both players are indifferent between {\em all} their pure strategies. It follows that if we consider the related game where the players have the same strategy sets but the Searcher is the maximizer and the Hider is the minimizer, $\nu$ and $\theta$ are still optimal for the Hider and Searcher, respectively, and $V_k$ is still the value of the game.

One interesting property of the optimal Hider strategy is that whatever strategy the Searcher uses, at any point of the game, after $j$ targets have been found, the remaining $k-j$ targets will be hidden optimally. To rephrase this in a more precise way, we write the optimal Hider strategy as $\nu(\mathcal{B},k)$, to emphasize the dependency on $\mathcal{B}$ and $k$. Then, after $j$ targets have been found in some subset $\mathcal{A} \subseteq \mathcal{B}$, the remaining targets will be located in $\mathcal{B} \setminus \mathcal{A}$ according to $\nu(\mathcal{B} \setminus \mathcal{A}, k-j)$. This is easy to verify, using the definition of $\nu$ (see \cite{lidbetter2013search} for details).  It follows that even if the Searcher is permitted to adapt her search to exploit information gathered during the search, this does not advantage her.

If the Hider is also allowed to play adaptively by moving the targets during the search, it was also shown in \cite{lidbetter2013search} that the non-adaptive optimal strategies remain optimal. However, for the fourth and final case, where only the Hider is allowed to play adaptively, the value of the game is higher, in general.

It follows from standard  linear programming theory that there must be an optimal strategy for the Searcher whose support size is no more than ${n \choose k}$, since this is the size of the Hider's strategy set. The optimal strategy described above mixes between ${n \choose k} (n-k)!$ pure strategies. It is an interesting open problem to find an optimal Searcher strategy in closed form whose support size is $\mathcal{O}(n^k)$.

\section{Concluding Remarks}

As indicated in Section~\ref{sec:non-equal}, more work needs to be done to understand the game in the case of non-equal detection probabilities. Also, not much is known for the case of multiple targets. For instance, can we generalize the results of Section~\ref{sec:multiple} to the case of imperfect detection?  What about if the Searcher wishes to find not all $k$ of the targets, but only $j$ of them, for some fixed $j=1,2,\ldots,k-1$? 

There are  almost endless variations of these games, and we invite the reader to dip their toe into the intruiging world of box searching.

\subsection*{Acknowledgement}

This material is based upon work supported by the Air Force Office of Scientific Research under award number FA9550-23-1-0556.

\bibliographystyle{apalike}
\bibliography{ref}

\begin{thebibliography}{}

\bibitem[Alpern, 2011]{alpern2011new}
Alpern, S. (2011).
\newblock A new approach to {G}al’s theory of search games on weakly
  {E}ulerian networks.
\newblock {\em Dynamic Games and Applications}, 1:209--219.

\bibitem[Alpern and Gal, 2003]{alpern2003theory}
Alpern, S. and Gal, S. (2003).
\newblock {\em The theory of search games and rendezvous}.
\newblock Kluwer, Boston.

\bibitem[Alpern and Lidbetter, 2013]{alpern2013mining}
Alpern, S. and Lidbetter, T. (2013).
\newblock Mining coal or finding terrorists: The expanding search paradigm.
\newblock {\em Operations Research}, 61(2):265--279.

\bibitem[Bellman, 1957]{bellman}
Bellman, R. (1957).
\newblock {\em Dynamic Programming}.
\newblock Princeton University Press, NJ.

\bibitem[Black, 1965]{black1965discrete}
Black, W.~L. (1965).
\newblock Discrete sequential search.
\newblock {\em Information and control}, 8(2):159--162.

\bibitem[Bram, 1963]{bram19632}
Bram, J. (1963).
\newblock A 2-player $n$-region search game. irm-31, operations evaluation
  group.
\newblock {\em Center for Naval Analysis}.

\bibitem[Clarkson and Lin, 2024]{clarkson2024computing}
Clarkson, J. and Lin, K.~Y. (2024).
\newblock Computing optimal strategies for a search game in discrete locations.
\newblock {\em INFORMS Journal on Computing}.

\bibitem[Clarkson et~al., 2023]{clarkson2023classical}
Clarkson, J., Lin, K.~Y., and Glazebrook, K.~D. (2023).
\newblock A classical search game in discrete locations.
\newblock {\em Mathematics of Operations Research}, 48(2):687--707.

\bibitem[Condon et~al., 2009]{condon2009algorithms}
Condon, A., Deshpande, A., Hellerstein, L., and Wu, N. (2009).
\newblock Algorithms for distributional and adversarial pipelined filter
  ordering problems.
\newblock {\em ACM Transactions on Algorithms (TALG)}, 5(2):1--34.

\bibitem[Ferguson, 2020]{ferguson2020course}
Ferguson, T.~S. (2020).
\newblock {\em A course in game theory}.
\newblock World Scientific.

\bibitem[Gal, 1979]{gal1979search}
Gal, S. (1979).
\newblock Search games with mobile and immobile hider.
\newblock {\em SIAM Journal on Control and Optimization}, 17(1):99--122.

\bibitem[Isaacs, 1965]{isaacs}
Isaacs, R. (1965).
\newblock {\em Differential Games}.
\newblock Wiley, New York.

\bibitem[Lidbetter, 2013]{lidbetter2013search}
Lidbetter, T. (2013).
\newblock Search games with multiple hidden objects.
\newblock {\em SIAM Journal on Control and Optimization}, 51(4):3056--3074.

\bibitem[Matula, 1964]{matula1964periodic}
Matula, D. (1964).
\newblock A periodic optimal search.
\newblock {\em The American Mathematical Monthly}, 71(1):15--21.

\bibitem[Queyranne and Schulz, 1994]{queyranne1994polyhedral}
Queyranne, M. and Schulz, A.~S. (1994).
\newblock {\em Polyhedral approaches to machine scheduling}.
\newblock Citeseer.

\bibitem[Roberts and Gittins, 1978]{roberts1978search}
Roberts, D. and Gittins, J. (1978).
\newblock The search for an intelligent evader: Strategies for searcher and
  evader in the two-region problem.
\newblock {\em Naval Research Logistics Quarterly}, 25(1):95--106.

\bibitem[Ruckle, 1991]{ruckle1991discrete}
Ruckle, W.~H. (1991).
\newblock A discrete search game.
\newblock In {\em Stochastic Games And Related Topics: In Honor of Professor LS
  Shapley}, pages 29--43. Springer.

\bibitem[Smith, 1956]{smith1956various}
Smith, W.~E. (1956).
\newblock Various optimizers for single-stage production.
\newblock {\em Naval Research Logistics Quarterly}, 3(1-2):59--66.

\end{thebibliography}

\end{document}